\documentclass[14pt]{article}
\usepackage{latexsym,amsfonts,epsfig,amssymb,amsmath,graphicx,amsthm}
\newtheorem{definition}{Definition}[subsection]
\newtheorem{example}{Example}[subsection]
\newtheorem{proposition}{Proposition}[section]
\newtheorem{corollary}{Corollary}[section]

\begin{document}
\author{~E.B.~Cohen, N.H.~Nguyen, J.~G.~Winde, \& A.A.~Yielding  \\ Eastern Oregon University}
\title{\textbf{INERTIA SETS FOR FAMILIES OF GRAPHS}}
\date{July 2, 2013} 
\maketitle

\begin{abstract}
\noindent
 This paper consists of a few results, discovered and proved during the 2012-2013 research group at Eastern Oregon University.    Inertia tables are a visual representation of the possible inertias of a given graph.  The inertia of a graph counts the number of real positive and negative eigenvalues of its corresponding adjacency matrix.  The problem of studying inertia tables is directly related to the inverse eigenvalue problem and can be used as a tool for the minimum rank problem.  This paper describes the inverse eigenvalue problem, and tools used.  We describe a number of new general formulas for various simple undirected graphs and improved upon an established notation for inertia tables.
\end{abstract}

\newtheorem{lemma}{Lemma}[section]
\newtheorem{thm}{Theorem}[section]

\section{Introduction}\label{sec:intro}
Inverse inertias of graphs were originally studied as a more feasible problem than computation of minimum rank, denoted $mr(G)$.  This work is discussed in the paper \textit{The Inverse Inertia Problem for Graphs} \cite{inertia}.  We were curious if inertias existed with ``indents" instead of the known charts with trapezoidal regions.  We have not found any yet. In this paper we found several formulas for inertia of various types of graphs.\newline \newline
Given a graph $G$, $S(G)$ is the set of all real symmetric $n \times n$ adjacency matrices, $A = [a_{ij}]$.  If $i \neq j$, there is a $0$ in this entry if there is no edge between vertices $i$ and $j$.  Otherwise, there is a non-zero entry.  A realization of this zero/non-zero pattern is when values have been selected for the non-zero entries.  Depending on which values are chosen, there are many adjacency matrices with different eigenvalues.  The inverse inertia problem doesn't ask which eigenvalues specifically are attainable, but how many positive ($\pi(A)$), negative ($\nu(A)$), and zero ($\delta(A)$) eigenvalues there are for $A$.  The ordered pair $(\pi(A),\nu(A))$ is called the partial inertia of $A$.  We can determine $\delta(A)$ by observing that $n - \pi(A) - \nu(A) = \delta(A)$.  The inertia of $G$ is the set of all possible ordered pairs of attainable eigenvalues. \newline \newline

\section{Definitions}\label{sec:def}

\subsection{Graph Definitions}

     An \textbf{undirected graph} is an ordered pair $G=(V, \varepsilon)$ with the following properties:
    \begin{enumerate}
        \item The first component, $V$, is a finite, non-empty set.  The elements of $V$ are called the \textit{\textbf{vertices}} of $G$.
        \item The second component, $\varepsilon$, is a finite set of sets.  Each element of $\varepsilon $ is a set that is comprised of exactly two distinct vertices.  The elements of $ \varepsilon$ are called the \textit{\textbf{edges}} of $G$.
    \end{enumerate}
    The \textit{\textbf{degree}} of a vertex $v$ is the number of edges incident to $v$.
    The \textit{\textbf{union}} of a graph $G$ with a graph $H$, denoted $G\bigcup H$, is the graph $(V(G)\bigcup V(H),\varepsilon(G)\bigcup\varepsilon(H))$.
    If G and H are graphs on at least two vertices, each with a vertex labeled $v$, then $G$ \textit{\textbf{join}} $H$ at vertex $v$, denoted $\displaystyle G\bigoplus_vH$, is the graph on $|G|+|H|-1$ vertices obtained by identifying the vertex $v$ in $G$ with the vertex $v$ in $H$.
    A \textit{\textbf{path}}, denoted $P_n$, is a string of $n$ vertices where each vertex is of degree $2$, except the endpoints, which are of degree $1$.
    A \textit{\textbf{cycle}}, denoted $C_n$, is a closed path on $n$ vertices.
    A \textit{\textbf{pendant path}} is a path in a graph $G$ such that one endpoint is connected by an edge to a vertex in G that is not included in the pendant path.
    A \textit{\textbf{generalized star}}, denoted $G_p$, is a graph composed of $p$ paths joined at a common vertex, the center vertex.
    A \textit{\textbf{bouquet}} is a graph composed of cycles joined at a common vertex, the center vertex.
    A \textit{\textbf{supernova}} is a graph consisting of a generalized star, $G_p$, and a bouquet, $H$, with center vertices labeled $v$ such that $\displaystyle G_p\bigoplus_vH$.
    A \textit{\textbf{pulsar}} is a graph consisting of two supernovas, $G$ and $H$ with center vertices $v$ and $u$ respectively, and a cycle, $C_n$, such that $\displaystyle G\bigoplus_vC_n\bigoplus_uH$ where $vu\in V(C_n)$ but $vu\notin E(C_n)$.
    A \textit{\textbf{binary star}} is a graph consisting of two supernovas, $G$ and $H$ with center vertices $v$ and $u$ respectively, and a path, $P_n$, such that $\displaystyle G\bigoplus_vP_n\bigoplus_uH$ where $vu\in V(P_n)$.
    A \textit{\textbf{bipartite graph}} is a graph whose vertex set can be split into sets $A$ and $B$ such that each edge is incident to one vertex in $A$ and one vertex in $B$. A \textit{\textbf{complete bipartite graph}}, denoted $K_{a,b}$ for $a=|A|$ and $b=|B|$, is a bipartite graph where every possible edge is present.

\subsection{Linear Algebra Definitions}
    Given a matrix $A$ and the scalar $\lambda$ and the nonzero vector $\mathbf{p}$ which satisfy $A\mathbf{p} = \lambda \mathbf{p}$, $\lambda$ is called an \textbf{eigenvalue} of $A$. The Eigenvalues are a special set of scalars associated with a linear system of equations and are called the characteristic roots.\\

    \noindent
    If $G$ is a graph with $n$ vertices, then its \textbf{adjacency matrix}, $A$, is an $n \times n$ matrix, where each row and column corresponds to a vertex of $G$. The element $a_{i,j} \in A$ specifies the number of edges incidence to vertex $i$ to vertex $j$. Note an adjacency matrix for an undirected graph is symmetric about the main diagonal. We only consider simple undirected graph in this paper so we label the entries of a realization with the weight of the edges, see example \ref{inertia table example}.\\

    \noindent
    Let $A$ be an $n \times n$ matrix, then we say A is symmetric if $A=A^T$.
    Given an undirected graph $G$, the \textbf{inertia} of $G$ is the ordered triple $(\pi(A),\nu(A),\delta(A))$, where $\pi(A)$ is the number of positive eigenvalues of A, $\nu(A)$ is the number of negative eigenvalues of $A$, and $\delta(A)$ is the multiplicity of $0$ as an eigenvalue of $A$. Also $\pi(A)+ \nu(A)+\delta(A)=n$ and $\pi(A)+\nu(A)=rank(A)$. The \textbf{Inertia Table} of $G$, denoted $I(G)$, is a set of points representing all the possible positive and negative eigenvalues.\\

    \noindent
    The \textbf{minimum rank line} of a graph $G$ consists of all points $(\pi(A),\nu(A))\in I(G)$ such that $\pi(A)+\nu(A)=mr(G)$.\\

    \noindent
    \begin{example}\label{inertia table example}
    The graph $P_3$, which is an example of $K_{1,2}$
    \begin{center}
    \includegraphics[width=.30\textwidth]{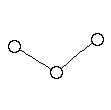}
    \end{center}
    can be represented by an adjacency matrix:
\[ \left[ \begin{array}{ccc}
r_1 & a & 0 \\
a & r_2 & b \\
0 & b & r_3 \end{array} \right]\] where $r_1,r_2,r_3 \in \mathbb{R}$, and $a,b \neq 0$.\\
Here is a realization of $P_3$:\\
\begin{center}
$A=$ $\left[
  \begin{array}{ccc}
    0 & 1 & 0 \\
    1 & 0 & 1 \\
    0 & 1 & 0 \\
  \end{array}
\right]$
\end{center}

This is a $3 \times 3$ symmetric matrix which has rank $2$ since $(\pi(A)=1, \nu(A)=1)$. Hence the inertia table correspond to this realization is:\\
\begin{center}
            \includegraphics[width=.40\textwidth]{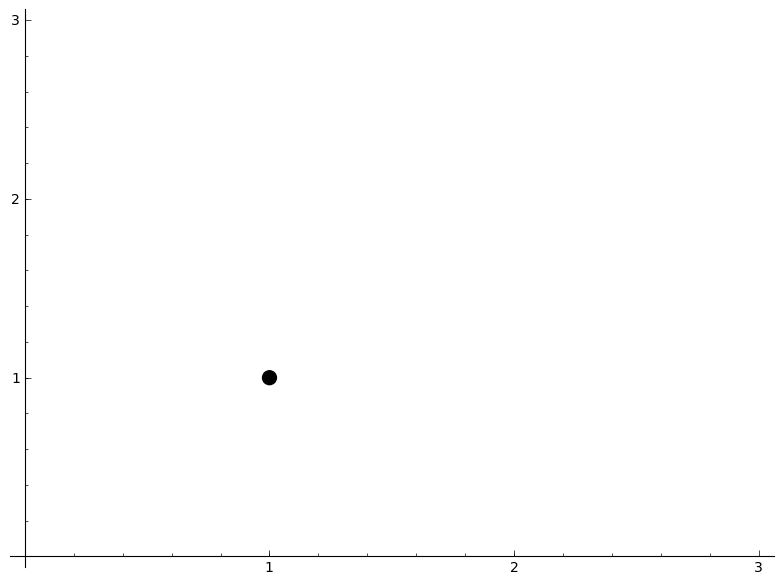}
\end{center}

\noindent
Other realizations of $P_3$:\\
\begin{center}
$B=$ $\left[
  \begin{array}{ccc}
    1 & 5 & 0 \\
    5 & 1 & 3 \\
    0 & 3 & 5 \\
  \end{array}
\right]$
\hspace{1cm}
$C=-$ $\left[
  \begin{array}{ccc}
    1 & 5 & 0 \\
    5 & 1 & 3 \\
    0 & 3 & 5 \\
  \end{array}
\right]$

\end{center}

\noindent
These realizations has rank $3$ since ($\pi(B)=2, \nu(B)=1)$ and ($\pi(C)=1, \nu(C)=2$).  Hence the inertia table for these realizations is:\\
\begin{center}
            \includegraphics[width=.40\textwidth]{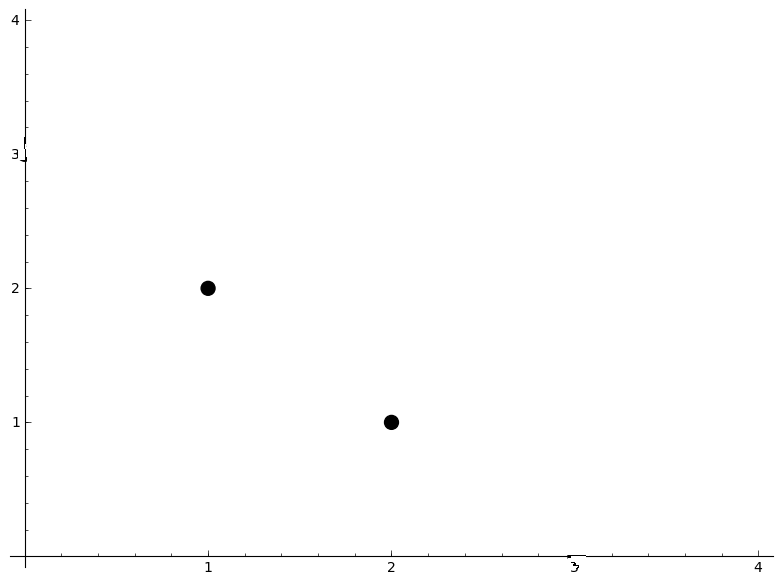}
\end{center}
\noindent
Therefore the inertia table for all realizations above is:\\
\begin{center}
        \includegraphics[width=.40\textwidth]{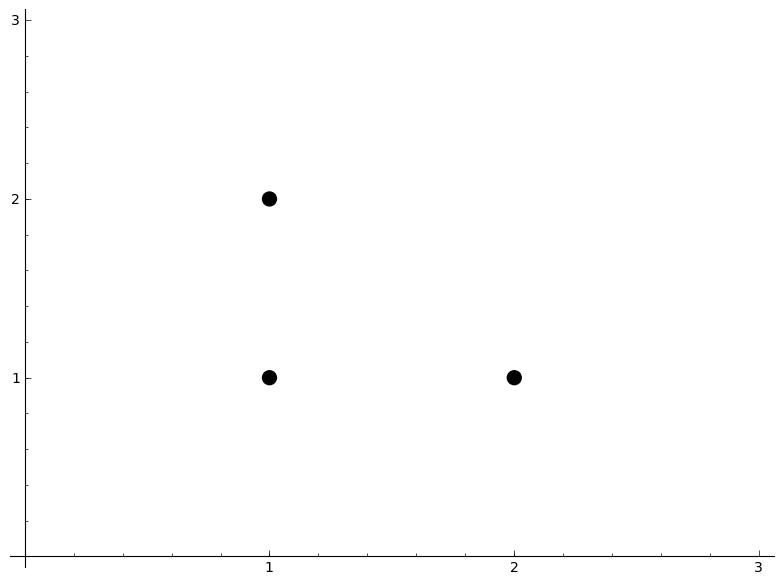}
\end{center}
\noindent
The inertia table above can be represent as $T^1_{[2,3]}$. Note that these are not the eigenvalues themselves, they are the number of negative and positive eigenvalues.\\
    \end{example}

\section{Notations and Conventions}\label{sec:notation}
In this section we introduce the $T$ notation for the inertia tables.  This is used to represent all the points on the inertia table.  Listing out all the points for an inertia table, or adding up two different inertias, can be tedious.  The $T$ notation allows us to do just that in a very simple and effective way.

\subsection{Understanding the $T$ notation}
The $T$ notation has the form $T_{[m,n]}^{k}$ for some $k$ and $m\leq n$, which are non-negative integers. The value $m$ represents the minimum rank, whereas $n$ represents the maximum, and $k$ indicates the inset from the axes. For convenience we often leave $k$ out of the notation when $k=0$. \\

\noindent
\begin{example}
        From example \ref{inertia table example} we saw $T^1_{[2,3]}$ represents the set of points:\\
	    $\{(1,1),(1,2),(2,1)\}$. Which has the inertia table below.\\
       \begin{center}
            \includegraphics[width=.40\textwidth]{I1}
        \end{center}
        \noindent
        Notice the inset from the axis is $1$, therefore $k=1$. The minimum rank is $2$ and the maximum rank is $3$.\\

        \noindent
        Similarly $T_{[2,3]}$ represents the set of points: $\{(0,2),(1,1),(2,0),(0,3),(1,2),(2,1),(3,0)\}$.\\
       \begin{center}
            \includegraphics[width=.40\textwidth]{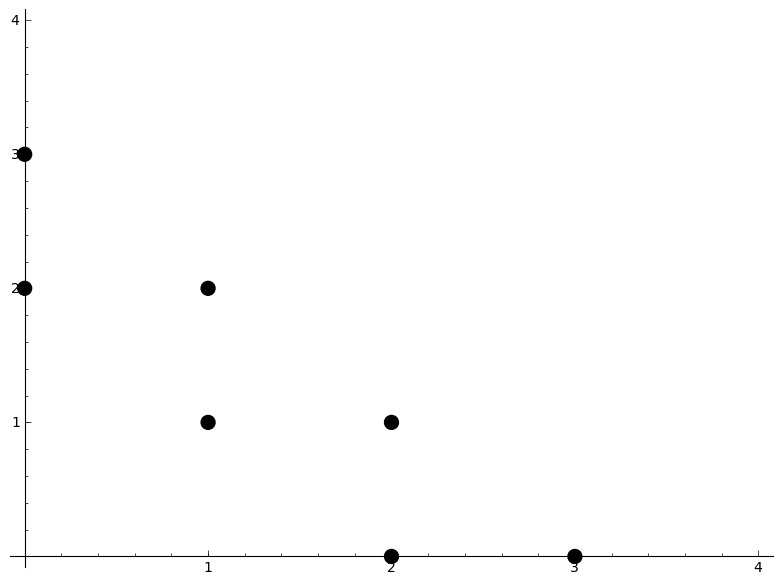}
        \end{center}
        \noindent
        Notice no inset from the axis, therefore $k=0$. Also observe this has minimum rank of $2$ and maximum rank of $3$.\\
\end{example}

  \noindent
    \begin{definition}\label{trapezoid}
    We say a graph $G$ has a \textbf{trapezoidal} inertia if $I(G)=T^{}_{[m,n]}$ for some non-negative integers $n$ and $m$.\\
    \end{definition}

\noindent
To represent an inertia that is not trapezoidal we must union several trapezoids together.  This can be represented as $\displaystyle T_{[m_1,n_1]}^{k_1} \bigcup T_{[m_2,n_2]}^{k_2} \bigcup T_{[m_3,n_3]}^{k_3} \bigcup ... \bigcup T_{[m_l,n_l]}^{k_l}$ $=\bigcup_{i=1}^l T^{k_i}_{[m_i,n_i]}$ for $i=0,1,2...l$.\\

\noindent
\begin{example}
    $\displaystyle T_{[3,4]} \bigcup T^1_{[2,2]}$ represents the set of points: \\
\{$(1,1),(0,3),(1,2),(2,1),(3,0),(0,4),(1,3),(2,2),(3,1),(4,0)$\}.  Note these points do not form a table of inertia with a trapezoidal shape because of the bump $(1,1)$.\\
    \begin{center}
    \includegraphics[width=.40\textwidth]{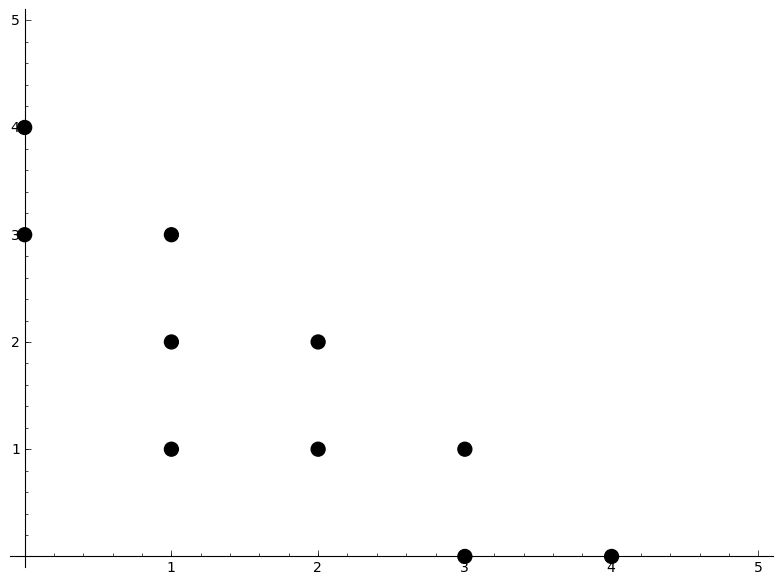}
    \end{center}
\end{example}

\subsection{Adding with $T$ notation}
In many cases we must add inertias. For example if you want to add $I(G_1) + I(G_2) +...+I(G_p)$ together, the computations can be very tedious when performing pointwise addition. Adding inertias using the $T$ notation helps to save time and minimize mistakes.\\
Suppose $I(G_i)=T^{k_i}_{[m_i,n_i]}$ for $i=0,1,2...,p$. Then $\displaystyle \sum^p_{i=0} I(G_i)$ $=I(G_1) + I(G_2) +...+I(G_p)$ $\displaystyle= T^{\sum^p_{i=0} k_i }_{[\sum^p_{i=0} m_i, \sum^p_{i=0} n_i]}$.\\

\noindent
\begin{example}
    Let $I(G_1)= T_{[2,3]}$, $I(G_2)=T^1_{[4,5]}$, $I(G_3)=T^2_{[3,6]}$ and $I(G_4)=T^3_{[4,8]}$.
    Then
    \begin{align*}
    I(G_1)+I(G_2)+I(G_3)+I(G_4) &=T_{[2,3]}+T^1_{[4,5]}+T^2_{[3,6]} + T^3_{[4,8]}\\
                                &=T^{0+1+2+3}_{[2+4+3+4,3+5+6+8]}\\
                                &=T^6_{[13,22]}
    \end{align*}
\end{example}

\noindent
Graphs with non-trapezoidal inertias are of the form $\displaystyle I(G)=\bigcup T^{k_j}_{[m_i,n_i]}$. So adding non-trapezoid inertia will require the use of distribution property within the union operator.
\noindent
\begin{example}
Let $\displaystyle I(G_1)= T_{[2,3]}\bigcup T_{[3,4]}$ and $\displaystyle I(G_2)=T_{[2,5]} \bigcup T_{[5,7]} $. Then
\begin{align*}
	I(G_1) + I(G_2)&= (T_{[2,3]}\bigcup T_{[3,4]})+ (T_{[2,5]} \bigcup T_{[5,7]})\\
	               &= (T_{[2,3]}+T_{[2,5]} \bigcup T_{[2,3]}+T_{[5,7]}) \bigcup (T_{[3,4]}+T_{[2,5]} \bigcup
                       T_{[3,4]}+T_{[5,7]})\\
                   &= (T_{[4,8]}\bigcup T_{[7,10]}) \bigcup (T_{[5,9]} \bigcup T_{[8,11]})\\
                   &= (T_{[4,7]}\bigcup T_{[7,10]}) \bigcup (T_{[5,8]} \bigcup T_{[8,11]})\\
                   &= T_{[4,10]}\bigcup T_{[5,11]}\\
                   &= T_{[4,11]}
	\end{align*}
Observed in the last three steps we moved the upper rank on some down since we do not want to have overlap in the notation. For instance in the fourth step we moved $T_{[4,8]}$ to $T_{[4,7]}$ and $T_{[5,8]}$ to $T_{[5,7]}$.

\end{example}

\section{Useful Tools}\label{sec:known theorems}
In this section we gather together several tools that allow us to compute the inertia tables for various types of graphs. All of these tools are from \textit{The Inverse Inertia Problem for Graphs} \cite{inertia}.

\indent

\begin{proposition}\label{disjoint}
    Let $\displaystyle G = \bigcup^k_{i=1} I(G_i)$. Then $I(G)= I(G_1) + I(G_2) + ...+I(G_k)= \sum^k_{i=1} G_i$ \cite{inertia}.\\
    \end{proposition}

    \begin{thm}\label{pathinertia}
    If $G$ is a graph such that $G\cong P_n$ then $I(G)= T^{}_{[n-1,n]}$ \cite{inertia}.\\
    \end{thm}

    \begin{thm}\label{cycleinertia}
    If $G$ is a graph such that $G \cong C_n$ then $I(G)= T^{}_{[n-2,n]}$ \cite{inertia}.\\
    \end{thm}

\noindent
Theorem \ref{The Theorem} is a significantly important theorem of finding inertia for graphs and is our main tool that allows us to come up with general formulas for the inertias of various graphs.

    \begin{thm}\label{The Theorem}
    Let $F$ and $G$ be graphs on at least two vertices with exactly one common vertex $v$ and let $n=|F| + |G|-1$. Then
    $\displaystyle I(F \bigoplus_v G)$ $= [I(F) + I(G)]_n$ $\bigcup$ $[I(F-v) + I(G-v) + {T^1_{[2,2]}}]_n$. \\
    \end{thm}

    \begin{example}
    Suppose $F=C_5$ and $G= P_3$.  $\displaystyle H = F \bigoplus_v G$ is the graph below; where $v$ is the degree $2$ vertex in $G$.
    \newline
    \begin{center}
     \includegraphics[width=.40\textwidth]{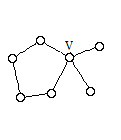}
    \end{center}
    \noindent
    Hence $I(F)= T^{}_{[3,5]}$ by theorem \ref{cycleinertia}
    \begin{center}
    \includegraphics[width=.40\textwidth]{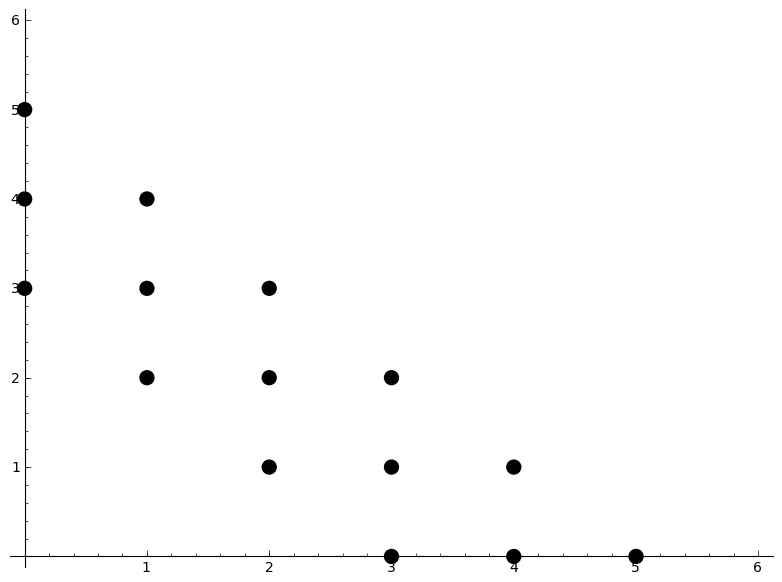}
    \end{center}
    Similarly, $I(G)=T^{}_{[2,3]}$ by theorem \ref{pathinertia}
    \begin{center}
    \includegraphics[width=.40\textwidth]{I2}
    \end{center}
    Thus
    \begin{align*}
    [I(F) + I(G)]&= T^{}_{[3,5]} + T^{}_{[2,3]}\\
                 &= T^{}_{[5,8]}
    \end{align*}
     but notice that the upper rank is $8$ which is absurd since the graph only has $7$ vertices and therefore the correspond adjacency matrix would be a $7 \times 7$. Therefore the maximum rank would be $7$ and to make this clear the maximum rank will be denote as subscript outside the bracket. Hence this can be written as $[I(F) + I(G)]_7 = T_{[5,7]}$.
    \begin{center}
    \includegraphics[width=.40\textwidth]{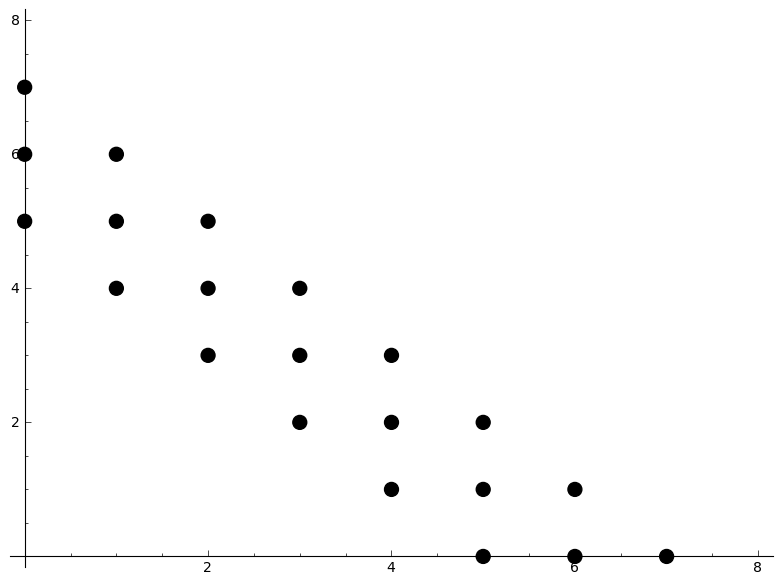}
    \end{center}
    \noindent

    Furthermore
    \begin{align*}
    I(F-v) &= I(C_5 - v)\\
           &= I(P_4)\\
           &=T^{}_{[4,5]}
    \end{align*}
    which has the inertia table below.
    \begin{center}
     \includegraphics[width=.40\textwidth]{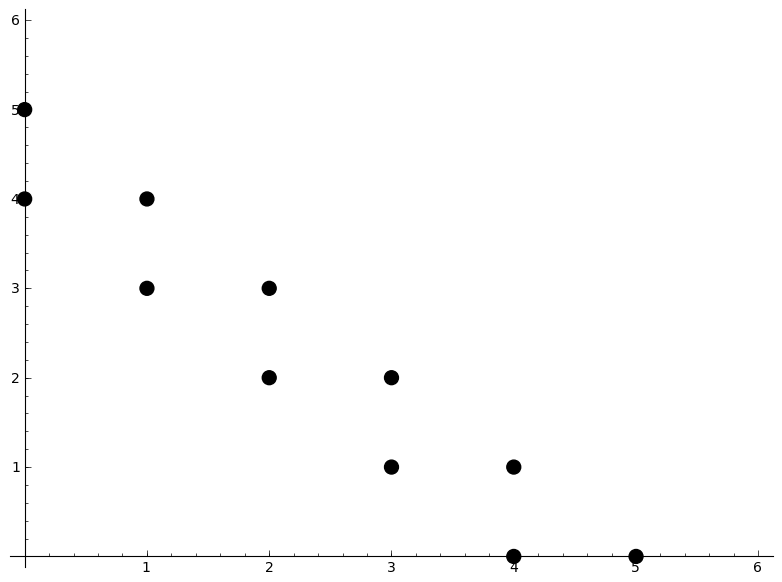}
    \end{center}
    \noindent

    Also
    \begin{align*}
    I(G-v)&=I(P_3-v)\\
          &=I(P_1 \bigcup P_1)\\
          &=I(P_1) + I(P_1)\\
          &= T^{}_{[0,1]} + T^{}_{[0,1]}\\
          &=T^{}_{[0,2]}.
    \end{align*}
    Note that here we implemented Proposition \ref{disjoint}.
    \begin{center}
    \includegraphics[width=.40\textwidth]{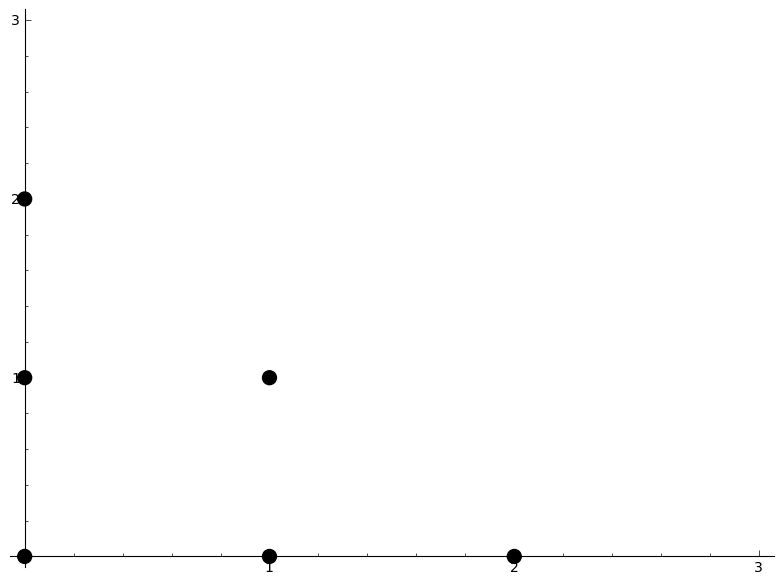}
    \end{center}
    \noindent
    Hence
    \begin{align*}
    [I(F-v) + I(G-v) + T^1_{[2,2]}]_7 &= [T^{}_{[4,5]} + T^{}_{[0,2]} + T^1_{[2,2]}]_7\\
                                       &= T^1_{[6,7]}
    \end{align*}
    \begin{center}
    \includegraphics[width=.40\textwidth]{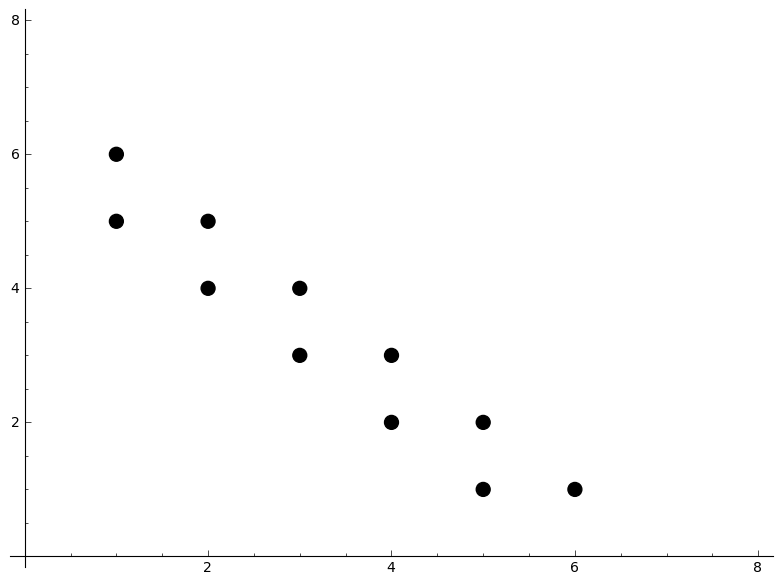}
    \end{center}
    \noindent
    Therefore
    \begin{align*}
    I(H) &=I(F \bigoplus_v G)\\
                                &=[I(C_5)+I(P_3)]_7  \bigcup   [I(P_4) + I(P_1 \bigcup P_1) + T_{[2,2]}]_7\\
                                &=T_{[5,7]} \bigcup T^1_{[6,7]}\\
                                &=T_{[5,7]}
    \end{align*}
    Thus the inertia table correspond to graph $H$ is:

    \begin{center}
    \includegraphics[width=.40\textwidth]{I8}
    \end{center}
    \noindent
    Observe that in this case $mr(H)=5$, and $I(H)$ is trapezoid.
    \end{example}

\section{Inertia Tables For Families Of Graphs }\label{sec:result1}
    We proceed with a series of lemmas that establish inertia of basic graphs, which then are used to prove theorem \ref{pulsar thm} and \ref{binary star thm}.

    \begin{lemma}\label{paths}
        If $G$ is a graph on $n$ vertices such that $G$ is the disjoint union of $k$ paths, then $I(G)=T^{}_{[n-k,n]}$
    \end{lemma}
    \begin{proof}
        Let $G$ be a graph on $n$ vertices such that $G$ is the disjoint union of $k$ paths.
        Label each path $P_{x_i}$, where $x_i$ is the number of vertices in the path $P_{x_i}$.
        So $\displaystyle \sum^{k}_{i=1}{x_i}=n$ and $I(G)=\displaystyle\sum^{k}_{i=1}{I(P_{x_i})}$. $\displaystyle \sum^{k}_{i=1}{I(P_{x_i})}=\sum^{k}_{i=1}{T^{}_{[x_i-1,x_i]}}=
        T^{}_{[\sum^{k}_{i=1}{(x_i-1)},\sum^{k}_{i=1}{x_i}]}=T^{}_{[\sum^{k}_{i=1}{(x_i)}-\sum^{k}_{i=1}{(1)},\sum^{k}_{i=1}{x_i}]}=T^{}_{[n-k,n]}$. Thus $\displaystyle I(G)=T^{}_{[n-k,n]}$
    \end{proof}

    \begin{lemma}\label{gs}
        If $G_p$ is a generalized star consisting of $n$ vertices , then $\displaystyle I(G_p)=T^{}_{[n-1,n]} \bigcup{T^{1}_{[n-p+1,n]}}$
    \end{lemma}
    \begin{proof}
        Suppose $G_p$ is a generalized star consisting of $n$ vertices and center $v$.
        Proceed by induction on $p$ the number of paths.
        Consider $G_1$ which has one path. Then $G_1$ $\cong{P_{n}}$, so $I(G_1)=T^{}_{[n-1,n]}$ = $T^{}_{[n-1,n]}$ $\displaystyle \bigcup{T^{1}_{[n-0,n]}}$.
        Now assume that $G_p$  has the inertia $I(G_p)=T^{}_{[n-1,n]}$ $\displaystyle \bigcup{T^{1}_{[n-p+1,n]}}$ for $p \leq m$.
        Now consider $G_{m+1}$, Notice
        $G_{m+1}\cong$ $\displaystyle H \bigoplus_v K$, where $H \cong G_m$ where $V|H|=x$ and $K$ $\cong{P_{y}}$ with $V|K|=y$.
        So
        $I(H)=T^{}_{[x-1,x]}$ $\displaystyle \bigcup{T^{1}_{[x-m+1,x]}}$ and $I(K)=T^{}_{[y-1,y]}$. Thus $[I(H) + I(K)]_n = T^{}_{[n-1,n]}$ $\displaystyle \bigcup T^{1}_{[n-m,n]}$.
        For $H-v$ which composed of $p$ disjoint paths. Hence $H-v$ $\cong{}$ $\displaystyle \bigcup^{m}_{i=1}{P_{n_{i}}}$ also $\sum{n_i}= x-1$ vertices. Thus $I(H-v)=T^{}_{[x-1-m,x]}$. Notice $K-v$ $\cong{P_{y-1}}$, hence $I(K-v)=T^{}_{[y-2,y]}$. Therefore $[I(H-v) + I(K-v) + T^{1}_ {[2,2]}]_n$
        = $T^{1}_{[n-m,n]}$. Thus $I(G_{m+1})=T^{}_{[n-1,n]}$ $\displaystyle \bigcup{T^{1}_{[n-m,n]} = T^{}_{[n-1,n]}}$ $\displaystyle \bigcup{T^{1}_{[n-(m+1)+1,n]}}$.
        Thus,the general inertia for $G_p$ is $I(G_p)=T^{}_{[n-1,n]}$ $\displaystyle \bigcup{T^{1}_{[n-p+1,n]}}$.
    \end{proof}

    \begin{lemma}\label{bouquet}
        If $G$ is a bouquet consisting of $n$ vertices and $k$ cycles, then $\displaystyle I(G)=T_{[n-k-1,n]}$.
    \end{lemma}
    \begin{proof}
        Let $G$ be a bouquet consisting of $n$ vertices and $k$ cycles joined at a common vertex $v$. Proceed by induction on $k$. Assume $G$ has one cycle, then $G\cong C_n$. Thus, $I(G)=I(C_n)=T_{[n-2,n]}=T_{[n-1-1,n]}$. Assume for $G$, a bouquet consisting of $k$ cycles, $I(G)=T_{[n-k-1,n]}$ for $1 \leq k \leq m$. Now consider $G$ a bouquet with $m+1$ cycles. Therefore, $\displaystyle G\cong F \bigoplus_v H$, where $F$ is a bouquet consisting of $x$ vertices and $m$ cycles and $H\cong C_y$, where $x+y=n+1$. So $\displaystyle I(G)=[I(F)+I(H)]_n\bigcup[I(F-v)+I(H-v)+T^1_{[2,2]}]_n$. $I(F)=T_{[x-m-1,x]}$ and $I(H)=T_{[y-2,y]}$. Thus $[I(F)+I(H)]_n=[T_{[x-m-1,x]}+T_{[y-2,y]}]_n=[T_{[x+y-m-3,x+y]}]_n=T_{[n-m-2,n]}$. $F-v$ is composed of $x-1$ vertices and $m$ disjoint paths. So $I(F-v)=T_{[x-m-1,x-1]}$. $H-v\cong P_{y-1}$ and $I(H-v)=T_{[y-2,y-1]}$. Therefore $\displaystyle [I(F-v)+I(H-v)+T^1_{[2,2]}]_n=[T_{[x-m-1,x-1]}+T_{[y-2,y-1]}+T^1_{[2,2]}]_n=[T^1_{[x+y-m-1,x+y]}]_n=T^1_{[n-m,n]}$. Thus $\displaystyle I(G)=T_{[n-m-2,n]}\bigcup T^1_{[n-m,n-1]}=T_{[n-m-2,n]}=T_{[n-(m+1)-1,n]}$. Hence the inertia of a bouquet with $n$ vertices and $k$ cycles is $I(G)= T_{[n-k-1,n]}$.
    \end{proof}

    \begin{lemma}\label{nova}
        If $G$ is a supernova on $n$ vertices consisting of $\alpha$ cycles and $\beta$ pendant paths, then $\displaystyle I(G)=T^{}_{[n-\alpha-1,n]}\bigcup T^1_{[n-\alpha-\beta+1,n]}$.
    \end{lemma}

    \begin{proof}
        Let $G$ be a supernova on $n$ vertices consisting of $\alpha$ cycles, $\beta$ pendant paths,and center vertex $v$.
        Then $\displaystyle G=F \bigoplus_v H$ where $F$ is a bouquet on $x$ vertices consisting of $\alpha$ cycles with center $v$, and $H$ is generalized star on $y$ vertices consisting of $\beta$ pendant paths, where $x+y=n+1$.
        So $\displaystyle I(G)=[I(F)+I(H)]_n\bigcup[I(F-v)+I(H-v)+T^1_{[2,2]}]_n$. $I(F)=T_{[x-\alpha-1,x]}$, and  $\displaystyle I(H)=T_{[y-1,y]}\bigcup T^1_{[y-\beta+1,y]}$.
        Thus $\displaystyle [I(F)+I(H)]_n=[T_{[x-\alpha-1,x]}+(T_{[y-1,y]}$
        $\displaystyle \bigcup T^1_{[y-\beta+1,y]})]_n=[T_{[x+y-\alpha-2,x+y]}$
        $\displaystyle\bigcup T^1_{[x+y-\alpha-\beta,x+y]}]_n =T_{[n-\alpha-1,n]}$ \newline $\displaystyle \bigcup T^1_{[n-\alpha-\beta+1,n]}$.
        Observe $F-v$ is composed of $x-1$ vertices and $\alpha$ disjoint paths, and $H-v$ is composed of $y-1$ vertices and $\beta$ disjoint paths.
        So $\displaystyle I(F-v)=T_{[x-\alpha-1,x-1]}$ and $I(H-v)=T_{[y-\beta-1,y-1]}$.
        Therefore $\displaystyle [I(F-v)+I(H-v)+T^1_{[2,2]}]_n$ $\displaystyle =[T_{[x-\alpha-1,x-1]}+T_{[y-\beta-1,y-1]}+T^1_{[2,2]}]_n$ $\displaystyle =[T^1_{[x+y-\alpha-\beta,x+y]}]_n=$
        $\displaystyle T^1_{[n-\alpha-\beta+1,n]}$.
        Thus $\displaystyle I(G)=(T_{[n-\alpha-1,n]}\bigcup T^1_{[n-\alpha-\beta+1,n]})\bigcup(T^1_{[n-\alpha-\beta+1,n]}) =T_{[n-\alpha-1,n]}$ \newline $\displaystyle\bigcup T^1_{[n-\alpha-\beta+1,n]}$
    \end{proof}

    \begin{thm}\label{pulsar thm}
        If $G$ is a pulsar on n vertices consisting of two supernovas with $\alpha_1$ cycles and $\beta_1$ pendant paths, and $\alpha_2$ cycles and $\beta_2$ pendant paths respectively, then $I(G)$ is: \newline $\displaystyle T_{[n-\alpha_1-\alpha_2-2,n]}\bigcup T^1_{[n-\alpha_1-\alpha_2-\beta_1,n]}\bigcup T^1_{[n-\alpha_1-\alpha_2-\beta_2,n]}\bigcup T^2_{[n-\alpha_1-\alpha_2-\beta_1-\beta_2,n]}$.
    \end{thm}
    \begin{proof}
        Let $G$ be a pulsar on n vertices consisting of two supernovas with $\alpha_1$ cycles and $\beta_1$ pendant paths, and $\alpha_2$ and cycles $\beta_2$ pendant paths respectively. Let $v$ be the central vertex of one of the supernovas. So $\displaystyle G \cong F \bigoplus_v H$ where $F$ is a supernova on $x$ vertices with $\alpha_1+1$ cycles and $\beta_1$ pendant paths, and $H$ is a supernova on $y$ vertices with $\alpha_2$ cycles and $\beta_2$ pendant paths, where $x+y=n+1$. So $\displaystyle I(G)=[I(F)+I(H)]_n\bigcup[I(F-v)+I(H-v)+T^1_{[2,2]}]_n$. Then $\displaystyle I(F)=T_{[x-\alpha_1-2,x]}$ $\displaystyle\bigcup T^1_{[x-\alpha_1-\beta_1,x]}$ and $\displaystyle I(H)=T_{[y-\alpha_2-1,y]}\bigcup T^1_{[x-\alpha_2-\beta_2+1,y]}$.\newline Thus $\displaystyle [I(F)+I(H)]_n=[(T_{[x-\alpha_1-2,x]}\bigcup T^1_{[y-\alpha_1-\beta_1,x]})+(T_{[y-\alpha_2-1,y]}$ \newline
        $\displaystyle \bigcup T^1_{[y-\alpha_2-\beta_2+1,y]})]_n=
        [T_{[x-\alpha_1-2+y-\alpha_2-1,x+y]}$
        $\displaystyle\bigcup T^1_{[x-\alpha_1-\beta_1+y-\alpha_2-1,x+y]}$ \newline
        $\displaystyle\bigcup T^1_{[x-\alpha_1-2+y-\alpha_2-\beta_2+1,x+y]}$
        $\displaystyle\bigcup T^2_{[x-\alpha_1-\beta_1+y-\alpha_2-\beta_2+1,x+y]}]_n=
        T_{[n-\alpha_1-\alpha_2-2,n]}$
        $\displaystyle\bigcup T^1_{[n-\alpha_1-\alpha_2-\beta_1,n]}$
        $\displaystyle\bigcup T^1_{[n-\alpha_1-\alpha_2-\beta_2,n]}$
        $\displaystyle\bigcup T^2_{[n-\alpha_1-\alpha_2-\beta_1-\beta_2+2,n]}$.
        $F-v$ is a supernova on $x-1$ vertices consisting of $\alpha_1$ cycles and $\beta_1+2$ pendant paths, and $H-v$ is composed of $y-1$ vertices and $\alpha_2+\beta_2$ disjoint paths. So $\displaystyle I(F-v)=T_{[x-1-\alpha_1-1,x-1]}$
        $\displaystyle\bigcup T^1_{[x-1-\alpha_1-\beta_1-2+1,x-1]}=T_{[x-\alpha_1-2,x-1]}$
        $\displaystyle\bigcup T^1_{[x-\alpha_1-\beta_1-2,x-1]}$ and $I(H-v)=T_{[y-\alpha_2-\beta_2-1,y-1]}$. Therefore $\displaystyle [I(F-v)+I(H-v)+T^1_{[2,2]}]_n=[(T_{[x-\alpha_1-2,x-1]}\bigcup T^1_{[x-\alpha_1-\beta_1-2,x-1]})+T_{[y-\alpha_2-\beta_2-1,y-1]}+T^1_{[2,2]}]_n$ \newline
        $\displaystyle=[T^1_{[x-\alpha_1-2+y-\alpha_2-\beta_2-1+2,x-1+y-1+2]}$
        $\displaystyle\bigcup T^2_{[x-\alpha_1-\beta_1-2+y-\alpha_2-\beta_2-1+2,x-1+y-1+2]}]_n$
        $\displaystyle =T^1_{[n-\alpha_1-\alpha_2-\beta_2,n]}\bigcup T^2_{[n-\alpha_1-\alpha_2-\beta_1-\beta_2,n]}$.
        Thus $\displaystyle I(G)=T_{[n-\alpha_1-\alpha_2-2,n+1]}$ \newline
        $\displaystyle\bigcup T^1_{[n-\alpha_1-\alpha_2-\beta_1,n+1]}\bigcup T^1_{[n-\alpha_1-\alpha_2-\beta_2,n+1]}$
        $\displaystyle\bigcup T^2_{[n-\alpha_1-\alpha_2-\beta_1-\beta_2+2,n+1]}$ \newline
        $\displaystyle\bigcup T^1_{[n-\alpha_1-\alpha_2-\beta_2,n+1]}\bigcup T^2_{[n-\alpha_1-\alpha_2-\beta_1-\beta_2,n+1]}$
        $\displaystyle = T_{[n-\alpha_1-\alpha_2-2,n]}$ \newline
        $\displaystyle\bigcup T^1_{[n-\alpha_1-\alpha_2-\beta_1,n]}\bigcup T^1_{[n-\alpha_1-\alpha_2-\beta_2,n]}\bigcup T^2_{[n-\alpha_1-\alpha_2-\beta_1-\beta_2,n]}$.
    \end{proof}

    \begin{thm}\label{binary star thm}
        If $G$ is a binary star on $n$ vertices. Where the pendant path shared by $H$ and $K$ is denoted as $P_w$
        then the inertias of $G$ have the following forms:
       \noindent
        \begin{enumerate}
        \item {For $w = 2$, $I(G)$ is \\
         $\displaystyle T^{}_{[n-\alpha-\delta-1,n]} \bigcup{T^{1}_{[n-\alpha-\gamma-\delta+1,n]}}$ $\displaystyle\bigcup{T^{1}_{[n-\alpha-\beta-\delta+1,n]}}$ $\displaystyle\bigcup{T^{2}_{[n-\alpha-\beta-\gamma-\delta+3,n]}}$.}\\
        \item {For $ w >2$,  $I(G)$ is\\
        $\displaystyle T^{}_{[n-\alpha-\delta-1,n]} \bigcup{T^{1}_{[n-\alpha-\gamma-\delta+1,n]}}$
        $\displaystyle \bigcup{T^{1}_{[n-\alpha-\beta-\delta+1,n]}}$
        $\displaystyle \bigcup{T^{2}_{[n-\alpha-\beta-\gamma-\delta+2,n]}}$.}
        \end{enumerate}

         Where $H$ is the supernova subgraph consisting of $\alpha$ cycles,and $\beta$ paths with center $v$,
         and $K$ is the supernovas subgraph combining of  $\delta$ cycles, and $\gamma$ paths with centers$u$.  $(v\neq u)$ and $v$,$u$ are connected by $P_w$ for $P_w\in \gamma$ paths of $K$.
    \end{thm}

    \begin{proof}
         Let $G$ be a binary star consisting of $n$ vertices such that $H$ and $K$ are Supernovas composed of $x$ vertices, $\alpha$ cycles, $\beta$ paths, $y$ vertices, $\delta$ cycles, $\gamma$ paths with center vertices $v$ and $u$ $(v\neq u)$  respectively. Proceed by implementing the the inertia formula from theorem \ref{The Theorem}.
         So $I(H)=T^{}_{[x-\alpha-1,x]}$ $\displaystyle\bigcup{T^{1}_{[x-\alpha-\beta+1,x]}}$ and $I(K)=T^{}_{[y-\delta-1,y]}$ $\displaystyle\bigcup{T^{1}_{[y-\delta-\gamma+1,y]}}$.
         Therefore $[I(H) + I(K)]_n = T^{}_{[n-\alpha-\delta-1,n]}$ $\displaystyle \bigcup{T^{1}_{[n-\alpha-\delta-\gamma+1,n]}}$ $\displaystyle\bigcup{T^{1}_{[n-\alpha-\beta-\delta+1,n]}}$
         $\displaystyle\bigcup{T^{2}_{[n-\alpha-\beta-\gamma-\delta+3,n]}}$.
         Observe $H-v$ is the collection of disjoint $\alpha + \beta$ paths consist of $x-1$ vertices. Hence, $I(H-v)$ $= T^{}_{[x-\alpha-\beta-1,x-1]}$.\\

         \noindent
         \textbf{Case I:} Here consider $ w=2 $. Hence $K-v$ is a supernova consisting of $y-1$ vertices with $\gamma-1$ paths and $\delta$ cycles. So $I(K-v) = T^{}_{[y-\delta-2,y-1]}$ $\displaystyle\bigcup{T^{1}_{[y-\delta-\gamma+1,y-1]}}$.
         Thus, $[I(H-v)+I(K-v)+T^{1}_{[2,2]}]_n$
         $=[T{}_{[n-\alpha-\beta-\delta-2,n]}$ $\displaystyle\bigcup T^{1}_{[n-\alpha-\beta-\delta-\gamma+1,n]}$ $ + T^{1}_{[2,2]}]_n$
         $ = T^{1}_{[n-\alpha-\beta-\delta,n]}$
         $\displaystyle\bigcup{T^{2}_{[n-\alpha-\beta-\gamma-\delta+3,n]}}$. Therefore \newline
         \noindent
         \begin{align*}
         I(G)&=
         (T^{}_{[n-\alpha-\delta-1,n]} \displaystyle\bigcup{T^{1}_{[n-\alpha-\delta-\gamma+1,n]}} \displaystyle\bigcup{T^{1}_{[n-\alpha-\beta-\delta+1,n]}}
         \displaystyle\bigcup{T^{2}_{[n-\alpha-\beta-\gamma-\delta+3,n]}})
         \end{align*}
         \indent \indent \indent
         $\displaystyle\bigcup(T^{1}_{[n-\alpha-\beta-\delta,n]}$ $\displaystyle\bigcup{T^{2}_{[n-\alpha-\beta-\gamma-\delta+3,n]}})$\\

         \begin{align*}
         \displaystyle &=T^{}_{[n-\alpha-\delta-1,n]} \displaystyle\bigcup{T^{1}_{[n-\alpha-\gamma-\delta+1,n]}} \displaystyle\bigcup{T^{1}_{[n-\alpha-\beta-\delta+1,n]}} \displaystyle\bigcup{T^{2}_{[n-\alpha-\beta-\gamma-\delta+3,n]}}.\\
         \end{align*}

         \noindent
         \textbf{Case II:} Here consider $ w>2 $. Thus $K-v$ is a Supernova  consisting of $y-1$ vertices and $\gamma$ paths. So
         $\displaystyle I(K-v) = T^{}_{[y-\delta-2,y-1]}\bigcup{T^{1}_{[y-\delta-\gamma,y-1]}}$.
         Hence, $[I(H-v)+I(K-v)+T^{1}_{[2,2]}]_n $
         $=[T{}_{[n-\alpha-\beta-\delta-2,n]}$
         $\displaystyle\bigcup$ $ T^{1}_{[n-\alpha-\beta-\delta-\gamma,n]}$ $ + T^{1}_{[2,2]}]_n$
         $=T^{1}_{[n-\alpha-\beta-\delta,n]}$
         $\displaystyle\bigcup{T^{2}_{[n-\alpha-\beta-\gamma-\delta+2,n]}} $. Therefore \newline
         \noindent
        \begin{align*}
        I(G)&=(T^{}_{[n-\alpha-\delta-1,n]}
        \displaystyle\bigcup{T^{1}_{[n-\alpha-\delta-\gamma+1,n]}} \displaystyle\bigcup{T^{1}_{[n-\alpha-\beta-\delta+1,n]}}
        \displaystyle\bigcup{T^{2}_{[n-\alpha-\beta-\gamma-\delta+3,n]}})
        \end{align*}
        \indent \indent \indent
        $\displaystyle\bigcup(T^{1}_{[n-\alpha-\beta-\delta,n]}$
        $\displaystyle\bigcup{T^{2}_{[n-\alpha-\beta-\gamma-\delta+2,n]}})$.\\
        \indent
        $= T^{}_{[n-\alpha-\delta-1,n]}
        \displaystyle\bigcup{T^{1}_{[n-\alpha-\gamma-\delta+1,n]}}
        \displaystyle\bigcup{T^{1}_{[n-\alpha-\beta-\delta+1,n]}}
        \displaystyle\bigcup{T^{2}_{[n-\alpha-\beta-\gamma-\delta+2,n]}}$.

    \end{proof}

\begin{corollary}\label{bipartite corollary}
    If $\displaystyle G = K_{a,b} \bigoplus_v K_{c,d}$, then $mr(G)$ does not exceed $4$.

\begin{proof}
    The inertia for a complete bipartite graph $K_{a,b}$, where $b \geq a$, is: \newline \newline
    $\displaystyle T^{}_{[b,a+b]} \bigcup T^{1}_{[2,b-1]}$ \newline \newline
    Using theorem \ref{The Theorem}, the inertia of $G$ has $4$ cases depending on the location of $v$: \newline \newline
    \textbf{Case I:} $ \left[ I(K_{a,b}) + I(K_{c,d}) \right]_n \bigcup \left[ I(K_{a-1,b}) + I(K_{c-1,d}) + T^{1}_{[2,2]} \right]_n $ \newline \newline
    This case has $v$ in the sets $A$ and $C$. \newline \newline
    \textbf{Case II:} $ \left[ I(K_{a,b}) + I(K_{c,d}) \right]_n \bigcup \left[ I(K_{a-1,b}) + I(K_{c,d-1}) + T^{1}_{[2,2]} \right]_n $ \newline \newline
    This case has $v$ in the sets $A$ and $D$. \newline \newline
    \textbf{Case III:} $ \left[ I(K_{a,b}) + I(K_{c,d}) \right]_n \bigcup \left[ I(K_{a,b-1}) + I(K_{c-1,d}) + T^{1}_{[2,2]} \right]_n $ \newline \newline
    This case has $v$ in the sets $B$ and $C$. \newline \newline
    \textbf{Case IV:} $ \left[ I(K_{a,b}) + I(K_{c,d}) \right]_n \bigcup \left[ I(K_{a,b-1}) + I(K_{c,d-1}) + T^{1}_{[2,2]} \right]_n $ \newline \newline
    This case has $v$ in the sets $B$ and $D$. \newline \newline
    These $4$ inertia cases each include
    $ I(K_{a,b}) + I(K_{c,d}) $ \newline \newline
    Thus,
    $ \left[ T^{}_{[b,a+b]} \bigcup T^{1}_{[2,b-1]} \right] + \left[ T^{}_{[d,c+d]} \bigcup T^{1}_{[2,d-1]} \right] $
    is included.  This is: \newline \newline
    $ T^{}_{[b+d,a+b+c+d-1]} \bigcup T^{1}_{[2+b,a+b+d-1]} \bigcup T^{1}_{[2+d,b+c+d-1]} \bigcup T^{2}_{[4,b+d-2]} $ \newline \newline
    Therefore the minimum rank line of $G$ does not exceed $4$.
\end{proof}

\end{corollary}

\section{Conclusion}\label{sec:conc}
We successfully discovered the inertia sets for several families of graphs. We also altered the standard T-notation, used to describe inertia tables, in order to more efficiently account for bumps that often occur when plotting inertia sets. \newline \newline
We have observed that the inertia tables for many graphs have bumps, where the minimum rank line is removed from the axes. However, is it possible for inertia tables to have indents? That is, is it possible for the minimum rank line to have points missing in the center while still touching the axes? During the writing of this paper it was still unknown if such graphs exist.

\end{document}